\documentclass{amsart}

\usepackage{latexsym}
\usepackage{amssymb}
\usepackage{amsmath}
\usepackage{mathrsfs}

\newtheorem{theorem}{Theorem}[section]

\newtheorem{proposition}[theorem]{Proposition}
\newtheorem{corollary}[theorem]{Corollary}

\theoremstyle{definition}
\newtheorem{definition}[theorem]{Definition}

\theoremstyle{remark}
\newtheorem{remark}[theorem]{Remark}
\newtheorem{example}[theorem]{Example}

\def\F{\mathscr{F}}
\def\G{\mathscr{G}}
\def\H{\mathscr{H}}

\def\PP{\mathcal{P}}
\def\Ps{\mathscr{P}}
\def\U{\mathscr{U}}
\def\V{\mathscr{V}}
\def\W{\mathscr{W}}
\def\X{\mathscr{X}}

\def\BB{\mathfrak{B}}
\def\FF{\mathfrak{F}}
\def\UU{\mathfrak{U}}

\def\E{\mathbb{E}}
\def\N{\mathbb{N}}

\def\R{\mathbb{R}}

\def\Fin{\text{Fin}}

\begin{document}

\title[Idempotent ultrafilters without Zorn's Lemma]
{Idempotent ultrafilters 
\\
without Zorn's Lemma}
\author{Mauro Di Nasso}
\author{Eleftherios Tachtsis}
\address{Dipartimento di Matematica\\
Universit\`a di Pisa, Italy}
\email{mauro.di.nasso@unipi.it}
\address{Department of Mathematics, University of the Aegean, Greece}
\email{ltah@aegean.gr}



\begin{abstract}
We introduce the notion of \emph{additive filter}
and present a new proof of the existence of idempotent ultrafilters
on $\N$ without any use of \emph{Zorn's Lemma}, and
where one only assumes the \emph{Ultrafilter Theorem} for 
the \emph{continuum}.
\end{abstract}

\maketitle

\medskip
\section*{Introduction}
Idempotent ultrafilters are a central object in Ramsey
theory of numbers. Over the last forty years they
have been extensively studied in the literature, 
producing a great amount of interesting combinatorial properties 
(see the extensive monography \cite{hs}). 
As reported in \cite{h05}, it all started in 1971 when F. Galvin 
realized that by assuming the existence of ultrafilters on $\N$
that are ``almost translation invariant'', one 
could produce a short proof of a conjecture of 
R. Graham and B. Rothschild, that was to become
a cornerstone of Ramsey theory of numbers:
``For every finite coloring of the natural numbers there
exists an infinite set $X$ such that all finite sums
of distinct elements of $X$ have the same color.''
However, at that time the problem was left open as whether 
such special ultrafilters could exist at all.
In 1972, N. Hindman \cite{h72} showed that the 
\emph{continuum hypothesis} suffices to construct 
those ultrafilters, but their existence in \textsf{ZFC}
remained unresolved. Eventually, in 1974, N. Hindman \cite{h74}
proved Graham-Rothschild conjecture (now known as Hindman's Theorem)
with a long and intricate combinatorial argument that avoided the use of 
ultrafilters. Shortly afterwards, in 1975, S. Glazer observed
that ``almost translation invariant''
ultrafilters are precisely the idempotent elements of the
semigroup $(\beta\N,\oplus)$,
where $\beta\N$ is the Stone-\v{C}ech compactification
of the discrete space $\N$ (which can be identified
with the space of ultrafilters on $\N$)
and where $\oplus$ is a suitable pseudo-sum operation
between ultrafilters. 
The existence of idempotent ultrafilters is then immediate, 
since any compact Hausdorff right-topological semigroup has idempotents,
a well-known fact in semigroup theory
known as \emph{Ellis-Numakura's Lemma}.
The proof of that lemma heavily relies on the axiom
of choice, as it consists in a clever and 
elegant use of \emph{Zorn's Lemma}
jointly with the topological properties of a compact Hausdorff space.
In 1989, T. Papazyan \cite{p} introduced the notion of
``almost translation invariant \emph{filter}'', and proved that the
maximal filters in that class, obtained by applying \emph{Zorn's Lemma},
are necessarily ultrafilters, and hence idempotent ultrafilters.

Despite their central role in a whole area of combinatorics
of numbers, no other proofs are known for the existence of idempotent 
ultrafilters. However, as it often happens with fundamental objects of 
mathematics, alternative proofs seem desirable because they may 
give a better insight and potentially lead to new applications.
It is worth mentioning that generalizations of idempotent 
ultrafilters have been recently considered both in the usual 
set-theoretic context, and in the general framework of model theory:
see \cite{k} where P. Krautzberger thoroughly investigated 
the almost translation invariant filters (appropriately
named ``idempotent filters''), and see \cite{ag}
where U. Andrews and I. Goldbring studied
a model-theoretic notion of idempotent type
and its relationship with Hindman's Theorem.

In this paper we introduce the notion of \emph{additive filter},
which is weaker than the notion of idempotent filter.
By suitably modifying the argument used in 
\emph{Ellis-Numakura's Lemma},
we show that \emph{Zorn's Lemma} is not needed to prove
that every additive filter can be extended to
a maximal additive filter, and that every maximal
additive filter is indeed an idempotent ultrafilter.
Precisely, we will only assume the following restricted form 
of the \emph{Ultrafilter Theorem} (a strictly weaker 
form of the axiom of choice):
``Every filter on $\R$ 
can be extended to an ultrafilter.''

\medskip
\section{Preliminary facts}

Although the notions below could also be considered
on arbitrary sets, here we will focus only on the set of 
natural numbers $\N$. We agree that a natural
number is a positive integer, so $0\not\in\N$. 

Recall that a \emph{filter} $\F$ is a nonempty family of nonempty sets
that is closed under supersets and under (finite) intersections.\footnote
{~More formally, $\F\subseteq\PP(\N)\setminus\{\emptyset\}$ is a filter if the following three properties
are satified: 

(1) $\N\in\F$, (2) $B\supseteq A\in\F\Rightarrow B\in\F$, (3)
$A,B\in\F\Rightarrow A\cap B\in\F$.}
An \emph{ultrafilter} is a filter that is maximal with respect to 
inclusion; equivalently, a filter $\U$ is an ultrafilter
if whenever $A\notin\U$, the complement $A^c\in\U$.
Trivial examples are given by the principal ultrafilters 
$\U_n=\{A\subseteq\N\mid n\in A\}$. Notice that an ultrafilter
is non-principal if and only if it extends the \emph{Fr\'echet filter}
$\{A\subseteq\N\mid A^c \text{ finite}\}$ of
cofinite sets. In the following,
$\F,\G$ will denote filters on $\N$, and $\U,\V,\W$
will denote ultrafilters on $\N$.

Recall that the \emph{Stone-\v{C}ech compactification} $\beta\N$
of the discrete space $\N$ can be identified
with the space of all ultrafilters on $\N$ endowed with the
Hausdorff topology that has the family
$\{\{\U\in\beta\N\mid A\in\U\}\mid A\subseteq\N\}$
as a base of (cl)open sets. (We note here that the above identification
is feasible in $\mathsf{ZF}$, \emph{i.e.}, in the Zermelo--Fraenkel set theory
minus the axiom of choice \textsf{AC}; see \cite[Theorems 14, 15]{hk}.)

The existence of non-principal ultrafilters is established
by the 

\smallskip
\begin{itemize}
\item
\emph{Ultrafilter Theorem} \textsf{UT}:
``For every set $X$, every proper filter on $X$ can be extended to an ultrafilter.''
\end{itemize}

\smallskip
The proof is a direct application of \emph{Zorn's Lemma}.\footnote
{~In the study of weak forms of choice, one usually
considers the equivalent formulation given by
the \emph{Boolean Prime Ideal Theorem} \textsf{BPI}:
``Every nontrivial Boolean algebra has a prime ideal''.
(See \cite{hr} where \textsf{BPI} is Form 14 
and \textsf{UT} is Form 14 A.)}
It is a well-known fact that \textsf{UT} is a strictly weaker 
form of \textsf{AC} (see, \emph{e.g.},
\cite{hr} and references therein). This means that 
one cannot prove \textsf{UT} in 
\textsf{ZF} alone, and that 
\textsf{ZF}+\textsf{UT} does not prove \textsf{AC}.


\begin{definition}
The \emph{pseudo-sum} 
of two filters $\F$ and $\G$
is defined by letting for every set $A\subseteq\N$:
$$A\in\F\oplus\G\ \Longleftrightarrow\ 
\{n\mid A-n\in\G\}\in\F$$
where $A-n=\{m\in\N\mid m+n\in A\}$ is
the rightward shift of $A$ by $n$.
\end{definition}

Notice that if $\U$ and $\V$ are ultrafilters,
then also their pseudo-sum $\U\oplus\V$ is an ultrafilter.
It is verified in a straightforward manner
that the space of ultrafilters $\beta\N$ endowed
with the pseudo-sum operation has the structure
of a \emph{right-topological semigroup}; that is, 
$\oplus$ is associative, and for every ultrafilter $\V$ 
the ``product-on-the-right'' $\U\mapsto\U\oplus\V$ 
is a continuous function on $\beta\N$ (see \cite{hs} for all details).

\begin{definition}
An \emph{idempotent ultrafilter}
is an ultrafilter $\U$ which is idempotent
with respect to the pseudo-sum operation, \emph{i.e.}, $\U=\U\oplus\U$.
\end{definition}

We remark that the notion of idempotent ultrafilter is
considered and studied in the general setting of semigroups
(see \cite{hs}; see also the recent book
\cite{t}); however, for simplicity,
here we will stick to idempotents in $(\beta\N,\oplus)$. 

For sets $A\subseteq\N$ and for ultrafilters $\V$, let us denote by
$$A_\V\ =\ \{n\mid A-n\in\V\}.$$
So, by definition,  $A\in \F\oplus\V$ if and only if $A_\V\in\F$.
Notice that for every $A,B$ one has $A_\V\cap B_\V=(A\cap B)_\V$,
$A_\V\cup B_\V=(A\cup B)_\V$, and
$(A_\V)^c=(A^c)_\V$.

The following construction of filters will
be useful in the sequel.

\begin{definition}
For filters $\F,\G$ and ultrafilter $\V$, let
$$\F(\V,\G)\ =\ \{B\subseteq\N\mid B\supseteq F\cap A_\V\text{ for some }
F\in\F\text{ and some }A\in\G\}.$$
\end{definition}

Notice that, whenever it satisfies the \emph{finite intersection property},
the family $\F(\V,\G)$ is 
the smallest filter that contains both $\F$ and $\{A_\V\mid A\in\G\}$.
Families $\F(\V,\G)$ satisfy the following properties 
that will be relevant to our purposes.

\begin{proposition}[\textsf{ZF}]\label{fvg}
Let $\F$ be a filter, and let $\V$ be an ultrafilter.
Then for every filter $\G\supseteq\F\oplus\V$,
the family $\F(\V,\G)$ is a filter such that
$\F\subseteq\F(\V,\G)$ and $\G\subseteq\F(\V,\G)\oplus\V$. 
\end{proposition}

\begin{proof}
The inclusion $\F(\V,\G)\supseteq\F$ follows from the trivial
observation that $\N_\V=\N$, and hence 
$F=F\cap\N=F\cap\N_\V\in\F(\V,\G)$
for every $F\in\F$.
All sets in $\F(\V,\G)$ are nonempty; indeed 
if $F\cap A_\V=\emptyset$ for some $F\in\F$ and some $A\in\G$,
then $F\subseteq(A_\V)^c=(A^c)_\V\Rightarrow (A^c)_\V\in\F
\Leftrightarrow A^c\in\F\oplus\V\subseteq\G\Rightarrow A\notin\G$.
Since $\F(\V,\G)$ is closed under supersets
and under finite intersections, it is a filter. Finally, 
$A\in\G\Rightarrow A_\V=\N\cap A_\V\in\F(\V,\G)\Leftrightarrow 
A\in\F(\V,\G)\oplus\V$.
\end{proof}

\begin{corollary}[\textsf{ZF}]\label{cor1}
Let $\F$ be a filter, and let $\V,\W$ be ultrafilters
where $\W\supseteq\F\oplus\V$.
Then for every ultrafilter $\U$ we have:
\begin{enumerate}
\item
If $\U\supseteq\F(\V,\W)$ then $\U\oplus\V=\W$\,;
\item
If $\U\oplus\V=\W$ and $\F\subseteq\U$ then 
$\F(\V,\W)\subseteq\U$.
\end{enumerate}
\end{corollary}

\begin{proof}
(1) If $\U\supseteq\F(\V,\W)$ then $\U\supseteq\F$ and
$\W\subseteq\F(\V,\W)\oplus\V\subseteq\U\oplus\V$,
and hence $\W=\U\oplus\V$, since no inclusion between
ultrafilters can be proper, by their maximality.

(2) By definition, $\U\oplus\V=\W$ if and only if $A_\V\in\U$ for all $A\in\W$.
Since $F\in\U$ for all $F\in\F$, it follows that $\F(\V,\W)\subseteq\U$.
\end{proof}

Let us now denote by $\textsf{UT}(X)$ the restriction of
\textsf{UT} to the set $X$, namely the property that
every filter on $X$ is extended to an ultrafilter.
In particular, in the sequel we will consider
$\textsf{UT}(\N)$ and $\textsf{UT}(\R)$.\footnote
{~$\textsf{UT}(\N)$ is Form 225 in \cite{hr}.}

\begin{corollary}[\textsf{ZF+UT}($\N$)]\label{cor2}
Let $\F$ be a filter and let $\V,\W$ be ultrafilters.
Then $\W\supseteq\F\oplus\V$ if and only if
$\W=\U\oplus\V$ for some ultrafilter $\U\supseteq\F$.
\end{corollary}

\begin{proof}
One direction is trivial, because $\U\supseteq\F$
directly implies $\U\oplus\V\supseteq\F\oplus\V$ for every $\V$.
Conversely, given an ultrafilter
$\W\supseteq\F\oplus\V$, by \textsf{UT}($\N$) we can
pick an ultrafilter $\U\supseteq\F(\V,\W)\supseteq\F$,
and the equality $\U\oplus\V=\W$ is satisfied by (1) of
the previous corollary.
\end{proof}

\medskip
\section{Additive filters}\label{sec-additive}

The central notion in this paper is the following.

\begin{definition}
A filter $\F$ is \emph{additive}
if for every ultrafilter $\V\supseteq\F$,
the pseudo-sum
$\F\oplus\V\supseteq\F$; that is, $A_\V\in\F$
for every $A\in\F$.
\end{definition}

\begin{remark}
In 1989, T. Papazyan \cite{p} considered the 
\emph{almost translation invariant filters} $\F$ such that 
$\F\subseteq\F\oplus\F$, and showed that every maximal filter 
in that class is necessarily an ultrafilter, and hence an idempotent ultrafilter. 
We remark that almost translation invariance is a stronger notion with
respect to addivity; indeed, it is straightforwardly seen that any 
almost translation invariant filter is additive, 
\emph{but not conversely}. (For the latter assertion, see
Example \ref{additivenotidem} in the next section.)
That same class of filters, named \emph{idempotent filters},
has been thoroughly investigated by P. Krautzberger in \cite{k}.
\end{remark}

A first trivial example of an additive filter is given by $\F=\{\N\}$;
another easy example is given by 
the Fr\'{e}chet filter $\{A\subseteq\N\mid A^c\text{ finite}\}$ of cofinite sets.
More interesting examples are obtained by considering
``additively large sets''.
For any $X\subseteq\N$,  
the set of all (finite) sums of distinct elements of $X$ is denoted by
$$\text{FS}(X) =\ \left\{\sum_{x\in F}x\,\Big|\,
F\subseteq X\text{ is finite nonempty}\right\}.$$

Recall that a set $A\subseteq\N$ is called 
\emph{additively large} if it contains
a set $\text{FS}(X)$ for some infinite $X$.
A stronger version of \emph{Hindman's Theorem}
states that the family of additively large
sets is \emph{partition regular}, \emph{i.e.}, if an additively
large set is partitioned into finitely many pieces, then
one of the pieces is still additively large.
By using a model-theoretic argument,
it was shown that this property is a $\textsf{ZF}$-theorem,
although no explicit proof is known
where the use of the axiom of choice is avoided
(see \S 4.2 of \cite{c}).

As mentioned in the introduction, idempotent ultrafilters
can be used to give a short and elegant proof
of Hindman's Theorem; indeed, in \textsf{ZF},
all sets in an idempotent
ultrafilter are additively large, whereas, in \textsf{ZFC},
for every additively large set $A$ there exists an
idempotent ultrafilter $\U$ such that $A\in\U$.\footnote
{~See Theorem 5.12 and Lemma 5.11 of \cite{hs}, respectively.}
For completeness, let us recall here a proof of the former 
combinatorial property, whose simplicity
and elegance was the main motivation 
for the interest in that special class of ultrafilters.

Notice first that if $\V\oplus\V=\V$ then
$\V$ is non-principal.\footnote
{~The only possible principal idempotent ultrafilter
would be generated by an element $m$ such that
$m+m=m$, whereas we agreed that $0\notin\N$.}
If $A\in\V$, then
$A^\star=A\cap A_\V\in\V$. It is readily verified
that $A^\star-a\in\V$ for every $a\in A^\star$. 
Pick any $x_1\in A^\star$. Then 
$A_1=A^\star\cap(A^*-x_1)\in\V$, and we can pick 
$x_2\in A_1$ where $x_2>x_1$.
Since $x_1,x_2,x_1+x_2\in A^\star$, the set 
$A_2=A^\star\cap(A^\star-x_1)\cap(A^\star-x_2)\cap(A^\star-x_1-x_2)\in\V$,
and we can pick $x_3\in A_2$ where $x_3>x_2$. 
By iterating the process,
one obtains an infinite sequence $X=\{x_1<x_2<x_3<\ldots\}$
such that $\text{FS}(X)\subseteq A^\star\subseteq A$, as desired.\footnote
{~For detailed proofs of other basic properties
of idempotent ultrafilters the reader is referred to \cite{hs}.}

Every additively large set determines an additive filter, as the
next \textsf{ZF}-example clarifies.

\begin{example}\label{fsx}
Given an infinite set $X=\{x_1<x_2<\ldots\}$, denote by
$$\mathcal{FS}_X\ =\ 
\left\{A\subseteq\N\mid A\supseteq\text{FS}(X\setminus F)
\text{ for some finite }F\subset X\right\}.$$
Clearly, $\mathcal{FS}_X$ is a filter that 
contains $\text{FS}(X)$. It just takes a quick check to verify that
$\mathcal{FS}_X\subseteq\mathcal{FS}_X\oplus\mathcal{FS}_X$,
and hence $\mathcal{FS}_X$ is additive.
\end{example}



Proposition \ref{char_addfil} below
provides a necessary and sufficient condition for a filter to be additive, and shows that
additive filters directly correspond to the closed 
sub-semigroups of $(\beta\N,\oplus)$.

\begin{proposition}[\textsf{ZF+UT}($\N$)]\label{char_addfil}
A filter $\F$ is additive if and only if 
$\F\subseteq\U\oplus\V$ for every pair of ultrafilters $\U,\V\supseteq\F$.
\end{proposition}

\begin{proof}
For the ``only if'' implication, notice first that \textsf{ZF+UT}($\N$)
implies that $\F$ equals the intersection of all
ultrafilters $\U\supseteq\F$. By the hypothesis, for every $A\in\F$
one has that $A_\V\in\U$ for all ultrafilters $\U\supseteq\F$,
and hence $A_\V\in\F$. Conversely,
assume there exists an ultrafilter $\V\supseteq\F$
with $\F\not\subseteq\F\oplus\V$, and pick
a set $A\in\F$ with $A\notin\F\oplus\V$, that is, $A_\V\notin\F$.
Then there exists an ultrafilter
$\U\supseteq\F$ with $A_\V\notin\U$
(note that the family $\{F\cap(A_\V)^c\mid F\in\F\}$ has
the finite intersection property, thus it can be extended to an ultrafilter
by $\textsf{UT}(\N)$).
Then $A\notin\U\oplus\V$ and the 
set $A$ is a witness of $\F\not\subseteq\U\oplus\V$.
\end{proof}

\begin{remark}\label{fil-cl}
If $C$ is a nonempty closed
sub-semigroup of $(\beta\N,\oplus)$ then
$$\text{Fil}(C)\ =\ \bigcap_{\U\in C}\U$$
is an additive filter.
To show this, notice first that if $\V\supseteq\text{Fil}(C)$
is an ultrafilter, then $\V\in \overline{C}=C$. So, for
all ultrafilters $\V,\V'\supseteq\text{Fil}(C)$ one has
that $\V\oplus\V'\in C$ by the property of sub-semigroup,
and hence $\V\oplus\V'\supseteq \text{Fil}(C)$.
Conversely, if $\F$ is an additive filter, then 
\textsf{UT}($\N$) implies that 
$$\text{Cl}(\F)\ =\ \{\U\in\beta\N\mid\U\supseteq\F\}$$
is a nonempty closed sub-semigroup. Moreover, the two operations
are one the inverse of the other, since
$\text{Cl}(\text{Fil}(C))=C$ and
$\text{Fil}(\text{Cl}(\F))=\F$ for every nonempty closed
sub-semigroup $C$ and for every additive filter $\F$.
\end{remark}

Next, we show two different ways of extending additive filters
that preserve the additivity property.

\begin{proposition}[\textsf{ZF+UT}($\N$)]\label{FplusVadditive}
Let $\F$ be an additive filter. Then for every ultrafilter $\V\supseteq\F$,
the filter $\F\oplus\V$ is additive.
\end{proposition}

\begin{proof}
Take any ultrafilter $\W\supseteq\F\oplus\V$. Then, by Corollary \ref{cor2},
there exists an ultrafilter $\U\supseteq\F$
such that $\W=\U\oplus\V$. By additivity of $\F$, we have that
$\F\subseteq\F\oplus\U\subseteq\V\oplus\U\Rightarrow
\F\subseteq\F\oplus\V\oplus\U\Rightarrow
\F\oplus\V\subseteq\F\oplus\V\oplus\U\oplus\V=
(\F\oplus\V)\oplus\W$.
\end{proof}

\begin{proposition}[\textsf{ZF}]\label{fvv}
Let $\F$ be an additive filter. Then for 
every ultrafilter $\V$ where $\V\supseteq\F\oplus\V$, 
$\F(\V,\V)$ is an additive filter.
\end{proposition}

\begin{proof}
Let $\U_1,\U_2\supseteq\F(\V,\V)$ be ultrafilters.
We want to show that $\F(\V,\V)\subseteq\U_1\oplus\U_2$.
Since $\F$ is additive and 
$\U_1,\U_2\supseteq\F(\V,\V)\supseteq\F$ by Proposition \ref{fvg},
we have that $\F\subseteq\U_1\oplus\U_2$. 
By Corollary \ref{cor1}, we have $\U_1\oplus\V=\U_2\oplus\V=\V$,
and so $\V=\U_1\oplus\U_2\oplus\V$.
But then for every $A\in\V$ the set $A_\V\in\U_1\oplus\U_2$,
and the proof is complete.
\end{proof}

\begin{theorem}[\textsf{ZF+UT}($\N$)]
If a filter is maximal among the additive filters
then it is an idempotent ultrafilter.
\end{theorem}

\begin{proof}
Let $\F$ be maximal among the additive filters.
By $\textsf{UT}(\N)$ we can pick an ultrafilter $\V\supseteq\F$.
We will show that $\V=\F$ and $\V\oplus\V=\V$.
By additivity $\F\subseteq\F\oplus\V$, and
since $\F\oplus\V$ is additive, by maximality
$\F=\F\oplus\V$. Since $\F$ is additive and the ultrafilter
$\V\supseteq\F\oplus\V$, also
the filter $\F(\V,\V)$ is additive by the previous
proposition and so, again by maximality, $\F(\V,\V)=\F$.
In particular, for every $A\in\V$ one has
that $A_\V\in\F(\V,\V)=\F$, that is
$A\in\F\oplus\V=\F$. This shows that $\V\subseteq\F$,
and hence $\V=\F$. Finally, since 
$\V\supseteq\F(\V,\V)$ and $\V\supseteq\F\oplus\V$, 
we have $\V\oplus\V=\V$ by Corollary \ref{cor1}.
\end{proof}

Thanks to the above properties of additive filters,
one proves the existence of idempotent ultrafilters
with a straight application of \emph{Zorn's Lemma}.

\begin{theorem}[\textsf{ZFC}]\label{maintheorem}
Every additive filter can be extended to an idempotent ultrafilter.
\end{theorem}

\begin{proof}
Given an additive filter $\F$, consider the following family
$$\mathbb{F}\ =\ \{\G\supseteq\F\mid\ \G\ \text{is an additive filter}\}.$$
It is easily verified that if $\langle \G_i\mid i\in I\rangle$ is an increasing
sequence of filters in $\mathbb{F}$,
then the union $\bigcup_{i\in I}\F_i$ is an additive filter.
So, \emph{Zorn's Lemma} applies, and one gets
a maximal element $\G\in\mathbb{F}$.
By the previous theorem, $\G\supseteq\F$ is an 
idempotent ultrafilter.
\end{proof}

\begin{remark}
As already pointed out in the introduction, with the only
exception of \cite{p}, the only known proof 
of existence of idempotent ultrafilters
is grounded on \emph{Ellis-Numakura's Lemma}, a general
result in semigroup theory that establishes the existence of idempotent 
elements in every compact right-topological semigroups.
An alternate argument to prove the above Theorem \ref{maintheorem}
can be obtained by same pattern.
Indeed, given an additive filter $\F$, by Remark \ref{fil-cl}
we know that $C=\text{Cl}(\F)$ is a closed nonempty sub-semigroup
of the compact right-topological semigroup $(\beta\N,\oplus)$.
In consequence, $(C,\oplus)$ is itself
a compact right-topological semigroup, so
\emph{Ellis-Numakura's Lemma} applies, and one gets
the existence of an idempotent
element $\U\in C$; clearly, $\U\supseteq\F$.
\end{remark}

As \emph{Zorn's Lemma} was never used in this section
except for the last theorem above, we are naturally lead to
the following question:

\smallskip
\begin{itemize}
\item
Can one prove Theorem \ref{maintheorem} without
using \emph{Zorn's Lemma}?
\end{itemize}

\smallskip
Clearly, at least some weakened form of the 
\emph{Ultrafilter Theorem} must be assumed, as otherwise
there may be no non-principal ultrafilters at all (see \cite{hr}). 
We will address the above
question in the next section.

\medskip
\section{Avoiding Zorn's Lemma}

\begin{proposition}[\textsf{ZF}]\label{step-one}
Assume there exists a choice function $\Phi$ that
associates to every additive filter $\F$
an ultrafilter $\Phi(\F)\supseteq\F$. Then
there exists a choice function $\Psi$ that
associates to every additive filter $\F$ an
ultrafilter $\Psi(\F)\supseteq\F$ such that
$\Psi(\F)\supseteq\F\oplus\Psi(\F)$.
\end{proposition}

\begin{proof}
Given an additive filter $\F$, let us define a sequence of filters
by transfinite recursion as follows.
At the base step, let $\F_0=\F$. At successor steps, 
let $\F_{\alpha+1}=\F_\alpha$ if
$\Phi(\F_\alpha)\supseteq\F_\alpha\oplus\Phi(\F_\alpha)$,
and let $\F_{\alpha+1}=\F_\alpha\oplus\Phi(\F_\alpha)$
otherwise. Finally, at limit steps $\lambda$, let 
$\F_\lambda=\bigcup_{\alpha<\lambda}\F_\alpha$.
It is readily seen by induction that all $\F_\alpha$
are additive filters, and that $\F_\alpha\subseteq\F_\beta$
for $\alpha\le\beta$.
If it was $\F_{\alpha+1}\ne\F_\alpha$
for all $\alpha$, then the sequence
$\langle\F_\alpha\mid\alpha\in\textbf{ON}\rangle$ 
would be strictly increasing.\footnote
{~By $\textbf{ON}$ we denote the proper class of all ordinals.} 
This is not possible, even without assuming $\textsf{AC}$,
because otherwise we would have a 1-1 correspondence
$\alpha\mapsto\F_{\alpha+1}\setminus\F_\alpha$
from the proper class $\textbf{ON}$
into the set $\mathcal{P}(\mathcal{P}(\N))$,
contradicting the replacement axiom schema.
Then define $\Psi(\F)=\Phi(\F_\alpha)$
where $\alpha$ is the least ordinal such that
$\F_{\alpha+1}=\F_\alpha$.
Such an ultrafilter $\Psi(\F)$ satisfies 
the desired properties. Indeed, 
$\Phi(\F_\alpha)\supseteq\F_\alpha\supseteq\F_0=\F$.
Moreover, if it was 
$\Phi(\F_\alpha)\not\supseteq\F\oplus\Phi(\F_\alpha)$,
then also
$\Phi(\F_\alpha)\not\supseteq\F_\alpha\oplus\Phi(\F_\alpha)$,
and so $\F_{\alpha+1}=\F_\alpha\oplus\Phi(\F_\alpha)$.
But then, since $\Phi(\F_\alpha)\supseteq\F_\alpha$
and $\Phi(\F_\alpha)\not\supseteq\F_{\alpha+1}$,
it would follow that
$\F_{\alpha+1}\ne\F_\alpha$, against the hypothesis.
\end{proof}

\begin{theorem}[\textsf{ZF}]\label{additivetoidempotent}
Assume there exists a choice function $\Phi$ that
associates to every additive filter $\F$
an ultrafilter $\Phi(\F)\supseteq\F$. Then
there exists a choice function $\Theta$ that
associates to every additive filter $\F$ an
idempotent ultrafilter $\Theta(\F)\supseteq\F$.
\end{theorem}

\begin{proof}
Fix a function $\Psi$ as given by the previous proposition.
Given an additive filter $\F$, by transfinite recursion 
let us define the sequence 
$\langle\F_\alpha\mid\alpha\in\textbf{ON}\rangle$ as follows.
At the base step, let $\F_0=\F$. 
At successor steps $\alpha+1$,
consider the ultrafilter $\V_\alpha=\Psi(\F_\alpha)$,
and let $\F_{\alpha+1}=\F_\alpha$ if 
$\V_\alpha\supseteq\F_\alpha(\V_\alpha,\V_\alpha)$, and 
let $\F_{\alpha+1}=\F_\alpha(\V_\alpha,\V_\alpha)$ otherwise.
At limit steps $\lambda$, let 
$\F_\lambda=\bigcup_{\alpha<\lambda}\F_\alpha$.
It is shown by induction that all $\F_\alpha$
are additive filters, and that $\F_\alpha\subseteq\F_\beta$
for $\alpha\le\beta$.
Indeed, notice that at successor steps
$\F_\alpha(\V_\alpha,\V_\alpha)\supseteq\F_\alpha$ is additive
by Proposition \ref{fvv}, since 
$\V_\alpha=\Psi(\F_\alpha)\supseteq\F_\alpha\oplus\V_\alpha$.
By the same argument as used in the proof of the
previous proposition, it cannot be
$\F_{\alpha+1}\ne\F_\alpha$ for all ordinals.
So, we can define $\Theta(\F)=\Psi(\F_\alpha)$
where $\alpha$ is the least ordinal such that
$\F_{\alpha+1}=\F_\alpha$.
Let us verify that the ultrafilter $\Theta(\F)$ satisfies 
the desired properties. 
First of all, $\Theta(\F)=\Psi(\F_\alpha)\supseteq\F$.
Now notice that
$\V_\alpha\supseteq\F_\alpha(\V_\alpha,\V_\alpha)$,
as otherwise 
$\F_{\alpha+1}=\F_\alpha(\V_\alpha,\V_\alpha)$
and we would have $\F_\alpha\ne\F_{\alpha+1}$,
since $\V_\alpha\supseteq\F_\alpha$ but
$\V_\alpha\not\supseteq\F_{\alpha+1}$.
So, $\Theta(\F)=\Psi(\F_\alpha)=\V_\alpha$,
and by Corollary \ref{cor1},
we finally obtain that $\V_\alpha\oplus\V_\alpha=\V_\alpha$.
\end{proof}

\begin{corollary}
[\textsf{ZF}]
Assume there exists a choice function $\Phi$ that
associates to every additive filter $\F$
an ultrafilter $\Phi(\F)\supseteq\F$. 
Then for every additively large set $A$ there
exists an idempotent ultrafilter $\U$ where $A\in\U$.
\end{corollary}

\begin{proof}
Let $X$ be an infinite set
with $\text{FS}(X)\subseteq A$, and consider
the additive filter $\mathcal{FS}_X$ of Example \ref{fsx}.
Then $A\in\mathcal{FS}_X$ and, by the previous theorem, 
$\mathcal{FS}_X$ is included in an idempotent ultrafilter.
\end{proof}

In order to prove that every additive filter extends to an
idempotent ultrafilter, one does not need the full
axiom of choice, and indeed we will see that a weakened version
of the \emph{Ultrafilter Theorem} suffices. 

The result below was proved 
in \cite[Lemma 4(ii)]{hkt} as the outcome of a chain of results about
the relative strength of $\textsf{UT}(\R)$ with respect
to properties of the Tychonoff
products $2^{\mathcal{P}(\R)}$ and
$2^\R$, where $2=\{0,1\}$ has the discrete topology.\footnote
{~Precisely, a proof of Proposition \ref{bpi} is obtained by combining the 
following \textsf{ZF}-results: (a) $\textsf{UT}(\R)$ if and only 
$2^{\mathcal{P}(\R)}$ is compact; 
(b) $\beta\N$ embeds as a closed subspace of 
$2^{\R}$; (c) $\textsf{UT}(\R)$ implies that 
$2^\R$, and hence $\beta\N$, is both compact and Loeb.
Recall that a topological space is \emph{Loeb} if there exists
a choice function on the family of its nonempty closed subspaces.
Recall also that nonempty closed subspaces of $\beta\N$ correspond to filters
(see Remark \ref{fil-cl}).}
In order to keep our paper self-contained, we give here
an alternative direct proof where explicit topological
notions are avoided.

\begin{proposition}[$\textsf{ZF+UT}(\R)$]\label{bpi}
There exists a choice function $\Phi$ that
associates to every filter $\F$ on $\N$
an ultrafilter $\Phi(\F)\supseteq\F$.
\end{proposition}

\begin{proof}
Every filter is an element of $\mathcal{P}(\mathcal{P}(\N))$,
which is in bijection with $\mathcal{P}(\R)$.
So, in $\textsf{ZF}$, one has a 1-1 enumeration of 
all filters $\{\F_Y\mid Y\in\FF\}$ for a suitable family
$\FF\subseteq\mathcal{P}(\R)$.
Fix a bijection $\psi:\mathcal{P}(\N)\times\mathcal{P}(\R)\to\mathcal{P}(\R)$,
let $I=\text{Fin}(\R)\times\text{Fin}(\text{Fin}(\R))$ (where for a set $X$, $\text{Fin}(X)$ denotes the set of finite subsets of $X$),
and for every $(A,Y)\in \mathcal{P}(\N)\times\mathcal{P}(\R)$, let
$$X(A,Y)\ =\ \{(F,S)\in I\mid S\subseteq\mathcal{P}(F);\ \psi(A,Y)\cap F\in S\}.$$
Notice that for every 
$\BB\subseteq\mathcal{P}(\N)\times\mathcal{P}(\R)$,
the family
$$\langle\BB\rangle\ =\ 
\{X(A,Y)\mid (A,Y)\in\BB\}\cup\{X(A,Y)^c\mid (A,Y)\notin\BB\}$$
has the finite intersection property.
Indeed, given pairwise distinct $(A_1,Y_1),\ldots$, $(A_k,Y_k)\in\BB$ 
and pairwise distinct $(B_1,Z_1),\ldots,(B_h,Z_h)\notin\BB$, 
for every $i,j$ pick an element
$u_{i,j}\in \psi(A_i,Y_i)\triangle\psi(B_j,Z_j)$ 
(where $\triangle $ denotes the symmetric difference of sets).
If we let 
$$F=\{u_{i,j}\mid i=1,\ldots,k;\  j=1,\ldots,h\}\quad\text{and}\quad
S=\{\psi(A_i,Y_i)\cap F\mid i=1,\ldots,k\},$$ 
then it is readily seen that $(F,S)\in \bigcap_{i=1}^{k}X(A_i,Y_i)\cap
\bigcap_{j=1}^{h} X(B_j,Z_j)^c$.

For every $Y\in\FF$, let us now consider the following
family of subsets of $I$:
$$\G_Y\ =\ \{X(A,Y)\mid A\in\F_Y\}\cup
\{\Lambda(A,B,Y)\mid A,B\subseteq\N\}\cup
\{\Gamma(A,Y)\mid A\subseteq\N\}$$
$$\cup\,\{\Delta(A,B,Y)\mid A\subseteq B\subseteq\N\}$$
where

\begin{itemize}
\item
$\Lambda(A,B,Y)=X(A,Y)^c\cup X(B,Y)^c\cup X(A\cap B,Y)$\,;
\item
$\Gamma(A,Y)=X(A,Y)\cup X(A^c,Y)$\,;
\item
$\Delta(A,B,Y)=X(A,Y)^{c}\cup X(B,Y)$.
\end{itemize}
We want to show that every finite union
$\bigcup_{i=1}^k\G_{Y_i}$ where $Y_i\in\FF$
has the finite intersection
property, and hence also $\G=\bigcup_{Y\in\FF}\G_Y$
has the finite intersection property.
By $\textsf{UT}(\N)$, 
which follows from $\textsf{UT}(\R)$,
we can pick ultrafilters $\V_i\supseteq \F_{Y_i}$ for $i=1,\ldots,k$.
Then
$$\H\ =\ \bigcup_{i=1}^k\left(\{X(A,Y_i)\mid A\in\V_i\}\cup
\{X(A,Y_i)^c\mid A\notin\V_i\}\right)$$
has the finite intersection property, because
$\H\subset\langle\BB\rangle$
where $\BB=\{(A,Y_i)\mid 1\le i\le k;A\in \V_i\}$.
Now let $G_1,\ldots,G_h\in\bigcup_{i=1}^k\G_{Y_i}$.
For every $G_j$ pick
$H_j\in\H$ such that $H_j\subseteq G_j$ as follows.
If $G_j=X(A,Y_i)$ for some $A\in\F_{Y_i}$
then let $H_j=G_j$; if $G_j=\Lambda(A,B,Y_i)$ then
let $H_j=X(A,Y_i)^c$ if $A\notin\V_i$, 
let $H_j=X(B,Y_i)^c$
if $A\in\V_i$ and $B\notin\V_i$,
and let $H_j=X(A\cap B,Y_i)$ if $A,B\in\V_i$;
if $G_j=\Gamma(A,Y_i)$ then let
$H_j=X(A,Y_i)$ if $A\in\V_i$,
and let $H_j=X(A^c,Y_i)$ if $A\notin\V_i$;
and if $G_{j}=\Delta(A,B,Y_{i})$ (where $A\subseteq B$) then let $H_j=X(B,Y_{i})$ if
$A\in\V_{i}$ or $B\in\V_{i}$, and let $H_{j}=X(A,Y_{i})^{c}$ if
$A\not\in\V_{i}$ and $B\not\in\V_{i}$.
But then $\bigcap_{j=1}^h G_j$ is nonempty
because it includes $\bigcap_{j=1}^h H_j$
and the family $\H$ has the finite intersection property.

Since $I=\text{Fin}(\R)\times\text{Fin}(\text{Fin}(\R))$ is
in bijection with $\R$, by $\textsf{UT}(\R)$ there exists an 
ultrafilter $\UU\supseteq\G$. Finally, 
for every $Y\in\FF$, the family 
$$\U_Y\ =\ \{A\subseteq\N\mid
X(A,Y)\in\UU\}$$
is an ultrafilter that extends $\F_Y$.
Indeed, if $A\in\F_Y$ then $X(A,Y)\in\G_Y\subseteq\UU$,
and so $A\in\U_Y$. Now assume $A,B\in\U_Y$, \emph{i.e.}
$X(A,Y),X(B,Y)\in\UU$. 
Since $\Lambda(A,B,Y)\in\G_Y\subseteq\UU$, we have
$X(A\cap B,Y)=\Lambda(A,B,Y)\cap X(A,Y)\cap X(B,Y)\in\UU$,
and so $A\cap B\in\U_Y$. Now let $A\in\U_{Y}$ and also let $B\supseteq A$.
Since $\Delta(A,B,Y)\in\G_{Y}\subseteq\UU$, we have $X(A,Y)\cap\Delta(A,B,Y)\in\UU$.
Furthermore, $X(B,Y)\supseteq X(A,Y)\cap X(B,Y)=
X(A,Y)\cap\Delta(A,B,Y)$, thus $X(B,Y)\in\UU$, and consequently $B\in\U_{Y}$.
Now let $A\subseteq\N$. If $A\notin\U_Y$, \emph{i.e.} if
$X(A,Y)\notin\UU$, then
$X(A,Y)^c\in\UU$. But $\Gamma(A,Y)\in\G_Y\subseteq\UU$,
so $X(A^c,Y)\supseteq\Gamma(A,Y)\cap X(A,Y)^c\in\UU$,
and hence $A^c\in\U_Y$. Clearly, the correspondence $\F_Y\mapsto\U_Y$
yields the desired choice function.
\end{proof}

\begin{remark}
\label{rem:Prop3_4}
In $\textsf{ZF}$, the property that ``there exists a choice function $\Phi$ that
associates to every filter $\F$ on $\N$
an ultrafilter $\Phi(\F)\supseteq\F$" is equivalent to 
the property that ``$\beta\N$ is compact and Loeb" 
(see \cite[Proposition 1(ii)]{hkt}). We remark that the
latter statement is strictly weaker than $\mathsf{UT}(\R)$ in 
$\mathsf{ZF}$ (see \cite[Theorem 10]{kk}).  
\end{remark} 

By putting together Proposition \ref{bpi} with
Theorem \ref{additivetoidempotent}, one obtains:

\begin{theorem}[$\textsf{ZF}+\textsf{UT}(\R)$]\label{main}
Every additive filter can be extended
to an idempotent ultrafilter.
\end{theorem}

\begin{remark}
Since every idempotent filter $\F\subseteq\F\oplus\F$ is readily seen to be
additive, as a straight corollary we obtain Paparyan's result \cite{p}
that every maximal idempotent filter is an idempotent ultrafilter.
\end{remark}

\begin{remark}
The above Theorem \ref{main}
cannot be proved by \textsf{ZF} alone.
Indeed, since the Fr\'{e}chet filter
$\{A\subseteq\N\mid \N\setminus A \text{ is finite}\}$
is additive, one would obtain the existence of a non-principal 
ultrafilter on $\N$ in \textsf{ZF}, against 
the well-known fact that there exist models of \textsf{ZF}
with no non-principal ultrafilters on $\N$ (see \cite{hr}).
\end{remark}

We conclude this section by showing an example of an additive filter
$\F$ which is not idempotent, \emph{i.e.}, $\F\not\subseteq\F\oplus\F$.

Recall that a nonempty family $\Ps\subseteq\PP(\N)$ is \emph{partition regular}
if in every finite partition $A=C_1\cup\ldots\cup C_n$ 
where $A\in\Ps$, one of the pieces $C_i\in\Ps$;
it also assumed that $\Ps$ is closed under supersets,
\emph{i.e.}, $A'\supseteq A\in\Ps\Rightarrow A'\in\Ps$.
In this case, the dual family 
$$\Ps^*\ =\ 
\{A\subseteq\N\mid A^c\notin\Ps\}\ =\ 
\{A\subseteq\N\mid A\cap B\ne\emptyset\ \text{ for every }B\in\Ps\}$$
is a filter; moreover, by assuming $\textsf{UT}(\N)$, one has 
$\Ps^*=\bigcap\{\U\in\beta\N\mid \U\supseteq\Ps\}$.
All these facts follow from the definitions in
a straightforward manner (see, \emph{e.g.},
\cite[Theorem 3.11]{hs} or \cite{bh}).

Call \emph{finitely additively large} (\emph{FAL} for short) 
a set $A\subseteq\N$ such that for every $n\in\N$ there exist
$x_1<\ldots<x_n$ with $\text{FS}(\{x_i\}_{i=1}^{n})\subseteq A$.
Clearly every additively large set is FAL, but not
conversely; \emph{e.g.}, the set 
$A=\bigcup_{k\in\N} \text{FS}(\{2^i\mid 2^{k-1}\le i<2^k\})$
is FAL but not additively large.\footnote
{~This example is mentioned in
\cite{bh2}, p. 4499; see also \cite[Theorem 1.12]{bh}, where 
FAL sets are called $\text{IP}_{<\omega}$ sets.}

\begin{example}\label{additivenotidem}
($\mathsf{ZF}$). The following family is an additive filter which is \emph{not} idempotent:
$$\F\ =\ \{A\subseteq\N\mid A^c\ \text{is not FAL}\}.$$

First of all, the dual family $\F$ is a filter because
the family of FAL sets is partition regular. Recall that the latter 
property is a consequence of \emph{Folkman's Theorem}
in its finite version:
``For every $n$ and for every $r$ there exists $N$
such that for every $r$-coloring $\{1,\ldots,N\}=C_1\cup\ldots\cup C_r$
there exists a set $S$ of cardinality $n$ with $\text{FS}(S)$ monochromatic.''
(For a $\textsf{ZF}$-proof of Folkman's Theorem, see
\cite[Theorem 11, Lemma 12]{grs}, pp. 81--82.)

In order to show that $\F$ is additive, fix any ultrafilter
$\V\supseteq\F$; we want to see that $\F\subseteq\F\oplus\V$.
Notice that every $B\in\V$ is FAL, as otherwise
$B^c\in\F\subseteq\V$ and we would have $\emptyset=B\cap B^c\in\V$.
By the definitions,
if $A\notin\F\oplus\V$ then
$A_\V=\{n\mid A-n\in\V\}\not\in\F$, \emph{i.e.},
$(A_\V)^c=\{n\mid A^c-n\in\V\}$ is FAL.
Then for every $n$ there exist 
$x_1<\ldots<x_n$ such that 
$A^c-s\in\V$ for every $s\in\text{FS}(\{x_i\}_{i=1}^{n})$.
Since the finite intersection 
$B=\bigcap\{A^c-s\mid s\in\text{FS}(\{x_i\}_{i=1}^{n})\}$ also belongs to $\V$,
we can pick $y_1<\ldots<y_n$ where $y_1>x_1+\ldots+x_n$
such that $\text{FS}(\{y_i\}_{i=1}^{n})\subseteq B$.
It is readily verified that $x_1+y_1<\ldots<x_n+y_n$
and that $\text{FS}(\{x_i+y_i\}_{i=1}^{n})\subseteq A^c$.
This shows that $A^c$ is FAL, and hence $A\not\in\F$.

Let us now check that the filter $\F$ is \emph{not} idempotent. 
To this end, we need some preliminary work.
Denote by $\N_0=\N\cup\{0\}$, and let $\psi:\Fin(\N_0)\to\N_0$ 
be the bijection where $\psi(F)=\sum_{i\in F}2^i$ for $F\ne\emptyset$
and $\psi(\emptyset)=0$.
Notice that $\psi(F)+\psi(G)=\psi(F\triangle G)+2\cdot\psi(F\cap G)$,
and also observe that for every $H$ one has
$2\cdot\psi(H)=\psi(1+H)$ where $1+H=\{1+h\mid h\in H\}$.
In consequence, if $\E$ is the set of even natural numbers, 
the following property is easily checked:
\begin{itemize}
\item[($\star)$]
\emph{Let $F,G\in\Fin(\E)$. Then $\psi(F)+\psi(G)=\psi(H)$ for some 
$H\in\text{Fin}(\E)$
if and only if $F\cap G=\emptyset$ and $F\cup G=H$.}
\end{itemize}

Now fix a partition $\E=\bigcup_{n\in\N_0}A_n$ of the
even natural numbers into infinitely many infinite sets, and define
$$X\ =\ \left\{\psi(F\cup G)\mid \emptyset\ne F\in\Fin(A_0)\ \&\
\emptyset\ne G\in\Fin(A_{\psi(F)})\right\}.$$
We will see that $X^c\in\F$ and $X^c\notin\F\oplus\F$,
thus showing that $\F\not\subseteq\F\oplus\F$.

The first property follows from the fact that there are no
triples $a,b,a+b\in X$, and hence $X$ is not FAL.
To see this, assume by contradiction that 
$\psi(F_1\cup G_1)+\psi(F_2\cup G_2)=\psi(F_3\cup G_3)$
for suitable nonempty  $F_1,F_2,F_3\in\Fin(A_0)$ and $G_i\in\Fin(A_{\psi(F_i)})$.
By the above property $(\star)$, it follows that
$(F_1\cup G_1)\cap(F_2\cup G_2)=\emptyset$, and hence
$F_1\cap F_2=\emptyset$; moreover,
$(F_1\cup G_1)\cup(F_2\cup G_2)=F_3\cup G_3$,
and hence $F_1\cup F_2=F_3$ and $G_1\cup G_2=G_3$.
This is not possible because $F_1\cap F_2=\emptyset$ implies that
$F_1,F_2\ne F_3$, and so 
$(G_1\cup G_2)\cap G_3=\emptyset$.

By the definitions, $X^c\not\in\F\oplus\F$ if and only if
$\Xi=\{n\mid X^c-n\in\F\}\not\in\F$ if and only if
$\Xi^c=\{n\mid X-n\ \text{is FAL}\}$ is FAL,
and this last property is true. Indeed, for every nonempty
$F\in\Fin(A_0)$ and for every nonempty $G\in\Fin(A_{\psi(F)})$,
we have that $F\cap G=\emptyset$ and so $\psi(F\cup G)=\psi(F)+\psi(G)$.
In consequence, the set 
$X-\psi(F)\supseteq\{\psi(G)\mid \emptyset\ne G\in\Fin(A_{\psi(F)})\}=
\text{FS}(A_{\psi(F)})$ is additively large, and hence FAL.
But then also $\Xi^c\supseteq\{\psi(F)\mid \emptyset\ne F\in\Fin(A_0)\}=
\text{FS}(A_0)$ is FAL because it is additively large, as desired.
\end{example}

\begin{remark}
The above example is fairly related to 
Example 2.8 found in P. Krautzberger's thesis \cite{k};
however there are relevant differences.
Most notably, besides the fact that different semigroups are considered, 
our example is carried within $\textsf{ZF}$, whereas the proof in
\cite{k} requires certain weak forms of the axiom of choice. 
Let us see in more detail.

In \cite{k} one first considers a \emph{partial} semigroup
$(\mathbb{F},\cdot)$ on the family $\mathbb{F}$ 
of finite subsets of $\N$ where the partial operation is 
defined by means of disjoint unions,
and then the corresponding semigroup 
of ultrafilters $(\delta\mathbb{F},\cdot)$ where 
$\delta\mathbb{F}$ is a suitable closed subspace of $\beta\mathbb{F}$.
(See \cite[Definition 1.4]{k} for details.)
By \emph{Graham-Rothschild parameter-sets Theorem} \cite{gr},
the family of sets that contain arbitrarily large finite union sets\footnote
{~A \emph{finite union set} is a set of the form
$\text{FU}(X)=\{\bigcup_{F\in\X} F\mid \emptyset\ne \X\in\text{Fin}(X)\}$.}
is partition regular, and so the following closed set is nonempty:
$$H\ =\ \{\U\in\delta\mathbb{F}\mid
(\forall A\in\U)(\forall n\in\N)(\exists x_1<\ldots <x_{n})\;
\text{FU}(\{x_{i}\}_{i=1}^{n})\subseteq A\}.$$ 
(Notice that $\textsf{UT}(\N)$ suffices to prove $H\ne\emptyset$; 
indeed, any ultrafilter on $\mathbb{F}$
extending the filter $\{A\subseteq\mathbb{F}\mid A^{c} 
\text{ does not contain arbitrarily large union sets}\}$
is in $H$.)
It is then shown that $H$ is a sub-semigroup
and that the filter 
$$\H\ =\ \text{Fil}(H)=\bigcap\{\U\mid\U\in H\}$$
is not idempotent. This last property is proved by showing the
existence of an injective sequence of ultrafilters 
$\langle\U_n\mid n\in\N\rangle$ in $H$ 
whose limit $\U$-$\lim_n(n+\U_n)\notin H$ for a suitable $\U\in H$;
notice that here countably many choices are made.\footnote
{~More precisely, for every $n\in\N$ one picks an ultrafilter $\U_n\in H$ 
that contains a suitable set $A_n\in\H$. (See \cite{k} for details.)
In view of Proposition \ref{bpi}, instead of
countable choice one could assume $\textsf{UT}(\R)$ 
to get such a sequence.}

It is well-known that partition regularity results about finite unions
can be (almost) directly translated into partition regularity results about finite sums,
and conversely (see, \emph{e.g.}, \cite[Theorem 13]{grs} and \cite[pp.113-114]{hs}).
Along these lines, our Example \ref{additivenotidem} can be
seen as a translation of the above example to $(\beta\N,\oplus)$.
We remark that, besides some non-trivial adjustments,
we paid attention not to use any form of choice; to this end,
we directly considered the dual filter 
$$\F\ =\ \{A\subseteq\N\mid A^c\ \text{is not FAL}\}.$$
instead of the corresponding closed sub-semigroup 
$\{\U\in\beta\N\mid(\forall A\in\U)(A\text{ is FAL})\}.$
\end{remark}

\section{Final remarks and open questions}

By only assuming a weaker property for a filter $\F$
than additivity,  one can prove
that every set $A\in\F$ is finitely additively large.

\begin{proposition}[\textsf{ZF}]\label{weak}
Let $\F$ be filter, and assume that there exists
an ultrafilter $\V\supseteq\F$ such $\F\subseteq\F\oplus\V$.
Then for every $A\in\F$ and for every $k$
there exist $k$-many elements $x_1<\ldots<x_k$ such
that $\mathrm{FS}(\{x_i\}_{i=1}^{n})\subseteq A$.
\end{proposition}

\begin{proof}
Let $\V$ be an ultrafilter as given by the hypothesis.
If $\V\supseteq\F$ is principal, say generated by $m\in\N$, then
$A\in\F\Rightarrow m\in A$.
Moreover, since $\F\subseteq\F\oplus\V$,
we also have that $A\in\F\Rightarrow A_\V=A-m\in\F$.
But then every $A\in\F$ contains all multiples $hm$ for $h\in\N$,
and the thesis trivially follows.
So, let us assume that $\V$ is non-principal.

For the sake of simplicity, here we will only consider the case $k=4$;
for arbitrary $k$, the proof is obtained by the same argument.
Notice first that, since $\F\subseteq\F\oplus\V$,
we have that $A\in\F\Rightarrow A_\V\in\F$, and hence
also $A_{\V^2},A_{\V^3}\in\F$, where we
denoted $\V^2=\V\oplus\V$ and $\V^3=\V\oplus\V\oplus\V$.

\begin{itemize}
\item
Pick $x_1\in A\cap A_\V\cap A_{\V^2}\cap A_{\V^3}\in\F$.
\end{itemize}

Then $x_1\in A$, and $A-x_1,A_\V-x_1,A_{\V^2}-x_1\in\V$.

\begin{itemize}
\item
Pick $x_2\in A\cap(A-x_1)\cap A_\V\cap (A_\V-x_1)
\cap A_{\V^2}\cap(A_{\V^2}-x_1)\in\V$.
As $\V$ is non-principal, we can take $x_2>x_1$.
\end{itemize}

Then $x_2,x_2+x_1\in A$, and
$A-x_2,A-x_1-x_2,A_\V-x_2,A_\V-x_1-x_2\in\V$.

\begin{itemize}
\item
Pick $x_3\in A\cap(A-x_1)\cap(A-x_2)\cap(A-x_1-x_2)\cap 
A_\V\cap (A_\V-x_1)\cap(A_{\V}-x_2)\cap(A_\V-x_1-x_2)\in\V$.
We can take $x_3>x_2$.
\end{itemize}

Then $x_3,x_3+x_1,x_3+x_2,x_3+x_2+x_1\in A$ and
$A-x_3,A-x_1-x_3,A-x_2-x_3,A-x_1-x_2-x_3\in\V$.

\begin{itemize}
\item
Pick $x_4\in A
\cap(A-x_1)
\cap(A-x_2)\cap(A-x_3)\cap
(A-x_1-x_2)\cap(A-x_1-x_3)\cap (A-x_2-x_3)
\cap(A-x_1-x_2-x_3)\in\V$.
We can take $x_4>x_3$.
\end{itemize}

We finally obtain that $\text{FS}(\{x_1<x_2<x_3<x_4\})\subseteq A$.
\end{proof}

Notice that, by combining Theorem \ref{main}
with the fact that every set in
an idempotent filter is additively large,
one obtains the following stronger property.

\begin{proposition}[$\textsf{ZF}+\textsf{UT}(\R)$]\label{additivehindman}
Let $\F$ be an additive filter. Then every $A\in\F$
is additively large.
\end{proposition}

Another corollary of Theorem \ref{main} is the following:

\begin{proposition}[$\textsf{ZF}+\textsf{UT}(\R)$]\label{extendingadditive}
Let $\F$ be an additive filter. Then for every $B\subseteq\N$
there exists an additive filter $\G\supseteq\F$ such that
either $B\in\G$ or $B^c\in\G$.
\end{proposition}

Since, in \textsf{ZF}, every element of an idempotent ultrafilter 
is additively large, it may be possible that the last two propositions
above are also \textsf{ZF}-results. With regard to this,
let us recall that also Hindman's Theorem is a theorem
of \textsf{ZF}, although
this was established only indirectly by a model-theoretic argument
(see \S 4.2 of \cite{c}), and as yet, no explicit 
\textsf{ZF}-proof of Hindman's Theorem is available.

\smallskip
\begin{enumerate}
\item
Is Proposition \ref{additivehindman}
provable in \textsf{ZF}?

\smallskip
\item
Is Proposition \ref{extendingadditive}
provable in \textsf{ZF}?
\end{enumerate}


\smallskip
Let us now consider the following statements:

\begin{enumerate}
\item[$(a)$]
``Every additive filter can be extended to an idempotent ultrafilter.''

\item[$(b)$]
``Every idempotent filter can be extended to an idempotent ultrafilter.''

\smallskip
\item[$(c)$]
``There exists an idempotent ultrafilter on $\N$.''

\smallskip
\item[$(d)$]
``There exists a non-principal ultrafilter on $\N$.''
\end{enumerate}
\smallskip

In the previous section, we showed in \textsf{ZF} that 
$\textsf{UT}(\mathbb{R})\Rightarrow (a)$ and noticed
that $(a)\Rightarrow(b)\Rightarrow(c)\Rightarrow(d)$.
We also recalled that $(d)$ cannot be proved by \textsf{ZF} alone.
These facts suggest to investigate whether
any of the above implications can be reversed.

\begin{enumerate}
\item[(3)]
Does \textsf{ZF} prove that $(a)\Rightarrow\textsf{UT}(\mathbb{R})$?

\smallskip
\item[(4)]
Does \textsf{ZF} prove that $(b)\Rightarrow(a)$?

\smallskip
\item[(5)]
Does \textsf{ZF} prove that $(c)\Rightarrow(b)$?

\smallskip
\item[(6)]
Does \textsf{ZF} prove that $(d)\Rightarrow(c)$?
\end{enumerate}


\begin{remark}
A detailed investigation of the strength of 
\emph{Ellis-Numakura's Lemma} in the hierarchy of weak choice 
principles is found in \cite{ta}.
In particular, in that paper it is shown that
either one of the \emph{Axiom of Multiple Choice} \textsf{MC}
or the \emph{Ultrafilter Theorem} \textsf{UT} (in its equivalent
formulation given by the \emph{Boolean Prime Ideal Theorem} \textsf{BPI})
suffices to prove Ellis-Numakura's Lemma.\footnote
{~\textsf{MC} postulates the existence of a 
``multiple choice'' function for 
every family $\mathcal{A}$ of nonempty sets,
\emph{i.e.}, a function $F$ such that
$F(x)$ is a nonempty finite subset of $x$ for every $x\in\mathcal{A}$.
Recall that \textsf{MC} is equivalent to \textsf{AC}
in \textsf{ZF}, but it is strictly weaker than \textsf{AC} in 
Zermelo--Fraenkel set theory with atoms \textsf{ZFA} (see \cite{hr}).}
(The key point of the proof is the fact that both
\textsf{MC} and \textsf{UT} imply the
existence of a choice function for the family of 
nonempty closed sub-semigroups
of any compact right topological semigroup.)
Recall that, as pointed out in Remark \ref{fil-cl}, 
under the assumption of \textsf{UT}($\N$) (or of \textsf{MC}, since \textsf{MC} $\Rightarrow$ \textsf{UT}($\N$)),
nonempty closed sub-semigroups of $(\beta\N,\oplus)$ exactly
correspond to additive filters, and so one obtains that
either one of \textsf{MC} or \textsf{UT} implies that
every additive filter on $\N$ is extended to an idempotent ultrafilter.
\end{remark}

\medskip
\bibliographystyle{amsalpha}

\end{document}